\documentclass[10pt,twocolumn,twoside]{IEEEtran}
\usepackage{graphicx}
\usepackage[usenames,dvipsnames,svgnames,table]{xcolor}
\definecolor{steelblue}{RGB}{70,130,180}
\usepackage[cmex10]{amsmath}
\usepackage{amssymb,amsfonts,mathrsfs,mathtools,amsthm}
\usepackage{epstopdf}
\IEEEoverridecommandlockouts

\newtheorem{remark}{Remark}
\newtheorem{definition}{Definition}
\newtheorem{theorem}{Theorem}
\newtheorem{corollary}{Corollary}

\DeclareMathOperator{\tr}{tr}

\DeclareMathOperator{\rk}{rank}

\newcommand{\dgp}[2]{%
\Delta(\,#1\,|\,#2\,)%
}

\newcommand*\dif{\mathop{}\!\mathrm{d}}

\title{On Submodularity and Controllability in \\Complex Dynamical Networks}
\author{Tyler H. Summers, Fabrizio L. Cortesi, and John Lygeros
\thanks{The authors are with the Automatic Control Laboratory, ETH Zurich, Switzerland, email: \{tsummers,jlygeros\}@control.ee.ethz.ch This work is partially supported by the ETH Zurich Postdoctoral Fellowship Program. }
\thanks{Preliminary versions of results in this paper appeared in \cite{summers2013optimal,cortesi2014submodularity}. Here, we present the work in a unified framework, provide modified and simplified proofs of the main results, and revise and elaborate on the numerical examples. }}
\date{\today}                                           

\begin{document}
\maketitle

\begin{abstract}
Controllability and observability have long been recognized as fundamental structural properties of dynamical systems, but have recently seen renewed interest in the context of large, complex networks of dynamical systems. A basic problem is sensor and actuator placement: choose a subset from a finite set of possible placements to optimize some real-valued controllability and observability metrics of the network. Surprisingly little is known about the structure of such combinatorial optimization problems. In this paper, we show that several important classes of metrics based on the controllability and observability Gramians have a strong structural property that allows for either efficient global optimization or an approximation guarantee by using a simple greedy heuristic for their maximization. In particular, the mapping from possible placements to several scalar functions of the associated Gramian is either a \emph{modular} or \emph{submodular set function}.  The results are illustrated on randomly generated systems and on a problem of power electronic actuator placement in a model of the European power grid.
\end{abstract}

\section{Introduction}

Controllability and observability have been recognized as fundamental structural properties of dynamical systems since the seminal work of Kalman in 1960 \cite{kalman1960contributions}, but have recently seen renewed interest in the context of large, complex networks, such as power grids, the Internet, transportation networks, and social networks. A prominent example of this recent interest is \cite{liu2011controllability}, which, based on Kalman's rank condition and the idea of structural controllability \cite{lin1974structural}, presents a graph theoretic maximum matching method to efficiently identify a minimal set of so-called driver nodes through which time-varying control inputs can move the system around the entire state space (i.e., render the system controllable). The method of \cite{liu2011controllability} is applied across a range of technological and social systems, leading to several interesting  and surprising conclusions. Using a metric of controllability given by the fraction of driver nodes in the minimal set required for complete controllability, it is shown that sparse inhomogeneous networks are difficult to control while dense homogeneous networks are easier. It is also shown that the minimum number of driver nodes is determined mainly by the degree distribution of the network. However, there is an implicit assumption about the diagonal entries of the dynamics matrix that restrict the application of their result \cite{cowan2012nodal}. Many other studies of controllability in complex networks have followed, including \cite{rajapakse2011dynamics,nepusz2012controlling,wang2012optimizing,tang2012epc,tang2012identifying}. 

One issue with the approach taken by \cite{liu2011controllability} and much of the follow-up work is that the exclusive focus on structural controllability and the associated quantitative notion of controllability (namely, the number/fraction of required driver nodes) can be rather crude and even misleading in some settings. This was noted, for example, by \cite{muller2011few} in response to the surprising result in \cite{liu2011controllability} that genetic regulatory networks seem to require many driver nodes, which apparently contradicts other findings in biological literature on cellular reprogramming. 
This suggests that rather than finding a set of driver nodes that would render a network completely controllable, a more appropriate strategy might be to choose from a finite set of possible actuator and sensor placements the subset that optimizes some real-valued controllability and observability metrics of the network. There is a variety of more sophisticated metrics for controllability and observability that have been proposed in the systems and control literature on sensor and actuator placement or selection problems in dynamical systems; see, e.g., the survey paper \cite{van2001review}. One important class of metrics involves the controllability and observability \emph{Gramians}, which are symmetric positive semidefinite matrices whose structure relates energy notions of controllability and observability. The use of Gramians as quantitative metrics of controllability in networks is studied in \cite{rajapakse2011dynamics,yan2012controlling,chapman2012system,sun2013controllability,pasqualetti2014controllability}. Other important studies of controllability in networks include \cite{sorrentino2007controllability,rahmani2009controllability,olshevsky2014minimal}.

While a variety of metrics have been proposed in the literature \cite{van2001review}, the corresponding combinatorial optimization problems for sensor and actuator placement are less well-understood. These can be solved by brute force for small problems by testing all possible placement combinations. However, for problems that arise in large networks, testing all combinations quickly becomes infeasible. Only very recently in the context of large networks have researchers started to investigate combinatorial properties of sensor and actuator placement problems for optimizing system dynamics and control metrics. 

Clark \emph{et al.} have recently considered a related but different problem of leader selection in networks with consensus dynamics, in which a set of leader states are selected to act as control inputs to the system \cite{bushnell2014supermodular,clark2011submodular,clark2012leader}. In \cite{bushnell2014supermodular,clark2011submodular}, it is shown that the minimum mean square error due to link noise is a supermodular function of the leader set. In \cite{clark2012leader} it is shown that a graph controllability index, which is related to the structural controllability framework of \cite{lin1974structural}, is a submodular function of the leader set. As discussed below, these properties allow for suboptimality guarantees using simple greedy algorithms. 




In the present paper, we show that one important class of metrics of controllability and observability, previously thought to lead to difficult combinatorial optimization problems \cite{van2001review}, can be in fact easily optimized, even for very large networks. In particular, we show that the mapping from subsets of possible actuator/sensor placements to any linear function of the associated controllability or observability Gramian has a strong structural property: it is a \emph{modular set function}. 
Furthermore, we show that the rank of the Gramian, the log determinant of the Gramian, and the negative trace of the inverse Gramian are \emph{submodular set functions}. Although maximization of submodular functions is difficult, submodularity allows for an approximation guarantee if one uses a simple greedy heuristic for their maximization \cite{greedybound}. 
We also describe how these observations define new dynamic network centrality measures for networks whose dynamics are described by linear models, assigning a control energy-related ``importance'' value to each node in the network. We illustrate the results on randomly generated systems and on a problem of power electronic actuator placement in a model of the European power grid.

The rest of the paper is organized as follows. Section II reviews the network model and Gramian-based controllability metrics. Section III introduces the notions of modular, submodular, and supermodular set functions and shows our main results that several set functions mapping possible actuator placements to various functions of the controllability Gramian are either modular or submodular. Section IV presents illustrative numerical examples. Finally, Section V gives concluding remarks and and outlook for future research.

\section{Linear Models of Network Dynamics}
This section defines a linear model for network dynamics and reviews and interprets metrics for controllability based on the controllability Gramian. The material in this section is mostly standard and can be found in many texts on linear system theory, e.g. \cite{kailath1980linear,callier1991linear}; we discuss the material mostly to set our notation. Since controllability and observability are dual properties \cite{kalman1959general}, we focus only on controllability and actuator placement; all of the results have analogous counterparts and interpretations for observability and sensor placement. 

In the literature on controllability in networks, it is common to start with linear network models. In this spirit, we focus on linear, time-invariant dynamical network models, in which the dynamics are given by
\begin{equation} \label{linearmodel}
\begin{aligned}
\dot{x}(t) &= A x(t) + B u(t), \quad x(0) = x_0, \\
         y(t) &= C x(t),
\end{aligned}
\end{equation}
where $x(t) \in \mathbf{R}^n$ are the states of the network, $u(t) \in \mathbf{R}^m $ are the control inputs that can be used to influence the network dynamics, $y(t) \in \mathbf{R}^p$ are the outputs, and $A$, $B$, and $C$ are constant matrices of appropriate dimensions. We assume that $C$ is full row rank. For example, $x(t)$ might represent voltages, currents, or frequencies in devices in a power grid, chemical species concentrations in a genetic regulatory network, or individual opinions or propensities for product adoption in a social network. 
The matrix $C \in \mathbf{R}^{p\times n}$ is typically interpreted as a set of linear state measurements, but here we will interpret it as a weight matrix whose rows define important directions in the state space.

The dynamics matrix $A$ induces a graph $\mathcal{G}(\mathcal{V},\mathcal{E})$ of the network in which the vertices correspond to states, i.e., $\mathcal{V} = \{ 1,...,n \}$ and the edges correspond to non-zero entries of $A$, i.e., $(i,j) \in \mathcal{E} $ whenever $a_{ji} \neq 0$. The non-zero entries of the input matrix $B$ describe how each actuator affects the nodes in the network. When optimizing actuator placement, one effectively designs a network structure by connecting sets of inputs to sets of network nodes to optimize a metric controllability for the resulting network. 


\subsection{Controllability}
\begin{definition}[Controllability]
A dynamical system is called \emph{controllable} over a time interval $[0,t]$ if given any states $x_0$, $x_1 \in \mathbf{R}^n$, there exists an input $u(\cdot):[0,t] \rightarrow \mathbf{R}^m$ that drives the system from $x_0$ at time $0$ to $x_1$ at time $t$. 
\end{definition}
Kalman's well-known rank condition states that a linear dynamical system is controllable if and only if $[B, AB, ..., A^{n-1} B]$ is full rank. Since rank is a generic property of a matrix, it has the same value for almost all values of the non-zero entries of $A$ and $B$ (assuming that the non-zero entries are independent). This suggests that controllability can be cast as a structural property of the graph defined by $A$ and $B$, as captured in the graph-theoretic concept of \emph{structural controllability} described by Lin in \cite{lin1974structural}, which underpins the recent results of \cite{liu2011controllability}. However, it is informative and practically relevant to consider more quantitative metrics for controllability in complex networks. 

\subsection{Energy-related controllability metrics}
Actuators in real systems are usually energy limited, so an important class of metrics of controllability deals with the amount of input energy required to reach a given state from the origin. 
The symmetric positive semidefinite matrix
\begin{equation} \label{gramint}
W_c(t) = \int_0^t e^{A \tau} B B^T e^{A^T \tau} d\tau \ \in \mathbf{R}^{n\times n}
\end{equation}
is called the \emph{controllability Gramian} at time $t$ and provides an energy-related quantification of controllability. 
Eigenvectors of $W_c$ associated with small eigenvalues (large eigenvalues of $W_c^{-1}$) define directions in the state space that are less controllable (require large input energy to reach), and eigenvectors of $W_c$ associated with large eigenvalues (small eigenvalues of $W_c^{-1}$) define directions in the state space that are more controllable (require small input energy to reach). 

For stable systems, the state transition matrix $e^{At}$ comprises decaying exponentials, so a finite positive definite limit of the controllability Gramian always exists and is given by
\begin{equation}
W_c = \int_0^\infty e^{A \tau} B B^T e^{A^T \tau} d\tau \ \in \mathbf{R}^{n\times n}.
\end{equation}
This infinite-horizon controllability Gramian can be computed by solving a Lyapunov equation
\begin{equation} \label{lyap}
A W_c + W_c A^T + B B^T = 0,
\end{equation}
which is a system of linear equations and is therefore easily solvable. Specialized algorithms have been developed to compute the solution \cite{bartels1972solution,hammarling1982numerical,li2002low} and scale to large networks.

We will focus on the infinite-horizon Gramian due to ease of computation. However, all of our results also apply to the finite-horizon Gramian, with the only disadvantage being that one must evaluate \eqref{gramint} rather than solve \eqref{lyap}, which may be more difficult for large networks. An advantage of the finite-horizon Gramian is that it can be used for unstable systems.

An alternative definition and interpretation of the Gramian for unstable systems can be used to quantify controllability \cite{balancedgramian}, but we do not discuss this interpretation further in the interest of simplicity. It is worth keeping in mind that though the results are stated for asymptotically stable systems, they apply more generally.

The controllability Gramian gives a more sophisticated energy-related quantitative picture of controllability, but we still need to form a scalar metric for $W_c$, which is a positive semidefinite matrix. We want $W_c$ ``large'' so that $W_c^{-1}$ is ``small'', requiring small amount of input energy to move around the state space. There are a number of possible metrics for the size of $W_c$, several of which we now discuss.

\subsubsection{$\mathbf{tr}(W_c)$}
The trace of the Gramian is inversely related to the average energy and can be interpreted as the average controllability in all directions in the state space. It is also closely related to the system $H_2$ norm:

\begin{equation}
\begin{aligned}
\Vert H \Vert_2^2  
			     &= \textbf{tr} \left( C \int_0^\infty e^{A t} B B^T e^{A^T t} dt C^T \right) \\
			     &= \textbf{tr} (C W_c C^T),
\end{aligned}
\end{equation}
i.e., the system $H_2$ norm is a weighted trace of the controllability Gramian.

%

\subsubsection{$\mathbf{tr}(W_c^{-1})$}
The trace of the inverse controllability Gramian is proportional to the energy needed on average to move the system around on the state space. Note that when the system is uncontrollable, the inverse Gramian does not exist and the average energy is infinite because there is at least one direction in which it is impossible to move the system using the inputs. In this case, one could consider the trace of the pseudoinverse, which is the average energy required to move the system around the controllable subspace. 


\subsubsection{$\log \det W_c$}
The determinant of the Gramian is related to the volume enclosed by the ellipse it defines
\[V(\mathcal{E}_{min}) = \frac{\pi^{n/2}}{\Gamma(n/2+1)}\sqrt[n]{\det W_c},\]
where $\Gamma$ is the Gamma function. This means that the determinant is a volumetric measure of the set of states that can be reached with one unit or less of input energy. Since determinant is numerically problematic in high dimensions, and because the logarithm is monotone, we will consider optimizing $\log \det W_c$. Note that for uncontrollable systems, the ellipsoid volume is zero, so $\log \det W_c = -\infty$. In this case, one could consider the associated volume in the controllable subspace. 



\subsubsection{$\lambda_{min} (W_c)$}

The smallest eigenvalue of the Gramian is a worst-case metric inversely related to the amount of energy required to move the system in the direction in the state space that is most difficult to control. 

\subsubsection{$\rk(W_c$)} The rank of the Gramian is the dimension of the controllable subspace.

\begin{remark}
Our main results and much of the discussion generalize straightforwardly to linear time-varying systems. The only differences are that the Gramian depends on both initial and final time, rather than just their difference, and that it must be computed by integrating \eqref{gramint}, rather than by solving a Lyapunov equation. 
\end{remark}

In the following section, we briefly review the combinatorial notion of submodularity and consider which of the above controllability metrics have a submodularity property, which provides global optimality or approximation guarantees for associated actuator selection problems.

\section{Optimal Sensor and Actuator Placement in Networks}

\subsection{Set Functions, Modularity, and Submodularity}

Sensor and actuator placement problems can be formulated as \emph{set function} optimization problems. For a given finite set $V = \{1,...,M \}$, a \emph{set function} $f: 2^V \rightarrow \mathbf{R}$ assigns a real number to each subset of $V$. In our setting, the elements of $V$ represent potential locations for the placement of actuators in a dynamical system, and the function $f$ is a metric for how controllable the system is for a given set of placements.

We consider set function optimization problems of the form
\begin{equation} \label{optprob}
 \underset{{S \subseteq V, \ |S| = k }}{\text{maximize}} \quad f(S).
\end{equation}
The problem is to select a $k$-element subset of $V$ that maximizes $f$. This is a finite combinatorial optimization problem, so one can solve it by brute force simply by  enumerating all possible subsets of size $k$, evaluating $f$ for all of these subsets, and picking the best subset. However, we are interested in cases arising from complex networks in which the number of possible subsets is very large. The number of possible subsets grows factorially as $|V|$ increases, so the brute force approach quickly becomes infeasible as $|V|$ becomes large. 


We focus instead on structural properties of the set function $f$ that make it more amenable to optimization. In particular, \emph{submodularity} plays a similar role in combinatorial optimization to convexity in continuous optimization and shares other features of concave functions \cite{lovasz1983submodular,krause2012submodular}. It occurs often in applications \cite{boykov2001interactive,kempe2003maximizing,krause2008near} (though is underexplored in systems and control theory); is preserved under various operations, allowing design flexibility; is supported by an elegant and practically useful mathematical theory; and there are efficient methods for minimizing and approximation guarantees for maximizing submodular functions. 
\begin{definition}[Submodularity]
A set function $f: 2^V \rightarrow \mathbf{R}$ is called submodular if for all subsets $A \subseteq B \subseteq V$ and all elements $s \notin B$, it holds that
\begin{equation} \label{submod1}
f(A \cup \{s\}) - f(A) \geq f(B \cup \{s\}) - f(B),
\end{equation}
or equivalently, if for all subsets $A,B \subseteq V$, it holds that
\begin{equation} \label{submod2}
f(A) + f(B) \geq f(A\cup B) + f(A\cap B).
\end{equation}
\end{definition}
Intuitively, submodularity is a diminishing returns property where adding an element to a smaller set gives a larger gain than adding one to a larger set. 
This is made precise by the following result from \cite{lovasz1983submodular}, which will be useful for verifying submodularity of set functions later.
\begin{definition}
A set function $f: 2^V \rightarrow \mathbf{R}$ is called monotone increasing if for all subsets $A, B \subseteq V$ it holds that
\begin{equation}
A \subseteq B \Rightarrow f(A) \leq f(B),
\end{equation}
and is called monotone decreasing if for all subsets $A, B \subseteq V$ it holds that
\begin{equation}
A \subseteq B \Rightarrow f(A) \geq f(B).
\end{equation}
\end{definition}
\begin{theorem}[\cite{lovasz1983submodular}] \label{theorem:mono}
A set function $f: 2^V \rightarrow \mathbf{R}$ is submodular if and only if the derived set functions $f_a : 2^{V - \{ a \} } \rightarrow \mathbf{R}$
$$f_a (X) = f(X \cup \{a\}) - f(X)  $$ 
are monotone decreasing for all $a \in V$.
\end{theorem}

A set function is called \emph{supermodular} if the reversed inequalities in (\ref{submod1}) and (\ref{submod2}) hold, and is called \emph{modular} if it is both submodular and supermodular, i.e., for all subsets $A,B \subseteq V$, we have $f(A \cap B) + f(A\cup B) = f(A) + f(B)$. A modular function has the following simple, equivalent characterization \cite{lovasz1983submodular}:
\begin{theorem}[Modularity \cite{lovasz1983submodular}] \label{modularity}
A set function $f: 2^V \rightarrow \mathbf{R}$ is modular if and only if for any subset $S \subseteq V$ it can be expressed as
\begin{equation} \label{submod3}
f(S) = w(\emptyset) + \sum_{s \in S} w(s)
\end{equation}
for some weight function $w : V \rightarrow \mathbf{R}$.
\end{theorem}
Modular set functions are analogous to linear functions and have the property that each element of a subset gives an independent contribution to the function value. Consequently, one can see that if $f$ is modular, the optimization problem \eqref{optprob} is easily solved by simply evaluating the set function for each individual element, sorting the result, and then choosing the top $k$ individual elements from the sorted list to obtain the best subset of size $k$. 

Maximization of monotone increasing submodular functions is NP-hard, but the so-called greedy heuristic can be used to obtain a solution that is provably close to the optimal solution \cite{greedybound}. The greedy algorithm for \eqref{optprob} starts with an empty set, $S_0 \leftarrow \emptyset$, computes the gain $\dgp{a}{S_i} = f(S_{i}\cup \{a\})-f(S_{i})$ for all elements $a\in V\backslash S_{i}$ and adds the element with the highest gain:
\[S_{i+1} \leftarrow S_{i}\cup \{\arg \max_{a} \dgp{a}{S_i}\; |\; a\in V\backslash S_{i}\}. \]
The algorithm terminates after $k$ iterations.

Performance of the greedy algorithm is guaranteed by a well known bound~\cite{greedybound}:
\begin{theorem}[\cite{greedybound}]
Let $f^*$ be the optimal value of the set function optimization problem \eqref{optprob}, and let $f(S_{greedy})$ be the value associated with the subset $S_{greedy}$ obtained from applying the greedy algorithm on \eqref{optprob}. If $f$ is submodular and monotone increasing, then 
\begin{equation}
	\frac{f^* - f(S_{greedy})}{f^* - f(\emptyset)}
	\leq \left(\frac{k-1}{k}\right)^k
	\leq \frac{1}{e}
	\approx 0.37.
\label{eq:greedy_bound}\end{equation}
\end{theorem}
This is the best any polynomial time algorithm can achieve \cite{feige1998threshold}, assuming $P\neq NP$. Note that this is a worst-case bound; the greedy algorithm often performs much better than the bound in practice. 

We now demonstrate the modularity or submodularity of several classes of controllability metrics involving functions of the controllability Gramian.

Recall that the space of symmetric $n\times n$ matrices $\mathcal{S}^n$ is partially ordered by the semidefinite partial order: $W_1\succeq W_2$ if  $W_1-W_2\succeq 0$. The space of symmetric positive definite matrices is denoted $\mathcal{S}_{++}^n$ and the space of symmetric positive semidefinite matrices is denoted $\mathcal{S}^n_{+}$. 

\subsection{Trace of the Gramian}
Suppose $A \in \mathbf{R}^{n\times n}$ is a stable system dynamics matrix and $V = \{b_1,..., b_M \} $ is a set of possible columns that can be used to form or modify the system input matrix $B$. The problem is to choose a subset of $V$ to maximize a metric of controllability. We now consider a linear function of the controllability Gramian, which can be expressed as a weighted trace of the controllability Gramian. For a given $S \subseteq V$, we form $B_S = [B_0 \quad b_s]$ given a (possibly empty) existing matrix $B_0$ and using the associated columns defined by  $s \in S$, and we denote the associated controllability Gramian $W_S  = \int_0^\infty e^{A \tau} B_S B_S^T e^{A^T \tau} d\tau$, which is the unique positive semidefinite solution the Lyapunov equation
\begin{equation}
AW_S + W_S A^T + B_S B_S^T = 0.
\end{equation}
To simplify notation, we write $W_s$ for $W_{\{s\} }$. We have the following result.
\begin{theorem} \label{contrmod}
The set function mapping subsets $S \subseteq V$ to a linear function of the associated controllability Gramian, $f(S) = \mathbf{tr}( C W_S C^T)$ for any weighting matrix $C \in \mathbf{R}^{p \times n}$, is modular.
\end{theorem}
\begin{proof} 
We will prove the result directly using Theorem \ref{modularity}. For any $S \subseteq V$ 
it is easy to see that the controllability Gramian associated with $B_S$ is simply a sum of the controllability Gramians associated with the individual columns of $B_S$:
\begin{equation} \label{gramadd}
\begin{aligned}
W_S  &= \int_0^\infty e^{A \tau} B_S B_S^T e^{A^T \tau} d\tau \\
           &= \int_0^\infty e^{A \tau} \sum_{s \in S} b_s b_s^T e^{A^T \tau} d\tau \\
           &= \sum_{s \in S} \int_0^\infty e^{A \tau} b_s b_s^T e^{A^T \tau} d\tau \\
           &= \sum_{s \in S} W_s .
\end{aligned}
\end{equation}
Now since trace is a linear matrix function, we have for any weight matrix $C \in \mathbf{R}^{p \times n}$
\begin{equation}
\begin{aligned}
f(S) &= \textbf{tr}( C W_S C^T ) \\
       &= \textbf{tr}\left(\sum_{s \in S} C W_{s} C^T \right) \\
       &= \sum_{s \in S} \textbf{tr} (C W_{s} C^T).
\end{aligned}
\end{equation}
Thus, for any $s \in V$, we can define the weight function $w(s) = \textbf{tr}(C W_s C^T)$. Defining $w(\emptyset) = 0$, Theorem \ref{modularity} implies that $f(S) = \textbf{tr}( C W_S C^T)$ is a modular set function.
\end{proof}
Theorem \ref{contrmod} shows that each possible actuator placement gives an independent contribution to the trace of the controllability Gramian. Because of this, the actuator placement problem using this metric is easily solved: one needs only to compute the metric individually for each possible actuator placement, sort the results, and choose the best $k$. Based on the interpretations in the previous section, this means that placing actuators in a complex network to maximize the average controllability available to move the system around the state space, or to maximize the energy in the system response to a unit impulse, is easily done. Since the result holds for the weighted trace, this gives considerable design freedom for actuator placement; important directions in the state space can be weighted and actuator placement done based on the weighted metric.

\subsection{Trace of the inverse Gramian}
We now consider the trace of the inverse of the controllability Gramian. We assume in this subsection that for any $S \subseteq V$ the associated Gramian $W_S$ is invertible. This is the case, for example, if the network already has a set of actuators that provide controllability and we would like to add additional actuators to improve controllability. We discuss how to deal with non-invertibility of the Gramian in Section III.E.

\begin{theorem}\label{theorem:trace}
Let $V = \{ b_1,...,b_M\}$ be a set of possible input matrix columns and $W_S$ the controllability Gramian associated with $S \subseteq V$. The set function $f : 2^{V}\to \mathbf{R}$ defined as
\[f(S) = -\tr(W_{S}^{-1}) \]
 is submodular and monotone increasing.
\end{theorem}
\begin{proof}
We will use Theorem \ref{theorem:mono} to prove the result. Fix an arbitrary $a \in V$ and consider the derived set function  $f_a: 2^{V-\{a\}} \to \mathbf{R}$ defined by
\begin{equation*}
\begin{aligned}
f_a(S) &=  -\tr((W_{S\cup \{a\} })^{-1}) + \tr((W_{S})^{-1}) \\
            &= -\tr((W_{S}+W_{a})^{-1}) + \tr((W_{S})^{-1}).
\end{aligned}
\end{equation*}
Take any $S_1 \subseteq S_2 \subseteq V- \{a\}$. By the additivity property of the Gramian noted in Theorem \ref{contrmod} in \eqref{gramadd}, it is clear that $S_1 \subseteq S_2 \Rightarrow W_{S_1} \preceq W_{S_2}$. Define $W(\gamma) = W_{S_1} + \gamma(W_{S_2} - W_{S_1})$ for $\gamma \in [0,1]$. Clearly, $W(0) = W_{S_1}$ and $W(1) = W_{S_2}$. Now define
\begin{equation*}
\hat{f}_a(W(\gamma)) =  -\tr((W(\gamma) + W_a )^{-1}) + \tr((W(\gamma) )^{-1}).
\end{equation*}
Note that $\hat{f}_a(W(0)) = f_a(S_1)$ and $\hat{f}_a(W(1)) = f_a(S_2)$. We have
\begin{equation*}
\begin{aligned}
&\frac{\dif}{\dif \gamma}\hat{f}_a\left(W(\gamma)\right) =\frac{\dif }{\dif \gamma}\left[- \tr((W(\gamma)+W_{a})^{-1}) + \tr(W(t)^{-1}) \right]\\
&=\tr\left[(W(\gamma)+W_{a})^{-1}(W_{S_2} - W_{S_1})(W(\gamma)+W_{a})^{-1} \right] \\
& \quad - \tr \left[W(\gamma)^{-1}(W_{S_2} - W_{S_1})W(\gamma)^{-1}\right]\\
&=\tr\bigg[ \left((W(\gamma)+W_{a})^{-2}-W(\gamma)^{-2}\right)(W_{S_2} - W_{S_1}) \bigg] \leq 0.
\end{aligned}
\end{equation*}
To obtain the second equality we used the matrix derivative formula $\frac{\dif}{\dif \gamma} X(\gamma)^{-1} =  X(\gamma)^{-1} \frac{\dif}{\dif \gamma}( X(\gamma)) X(\gamma)^{-1}$ \cite{petersen2008matrix}. To obtain the third equality we used the cyclic property of trace. Since $(W(\gamma)+W_{a})^{-2}-W(\gamma)^{-2} \preceq 0$ and $W_{S_2} - W_{S_1} \succeq 0$, the last inequality holds because  the trace of the product of a positive and negative semidefinite matrix is non-positive.
Since
\begin{equation*}
\hat{f}_a(W(1)) = \hat{f}_a(W(0)) + \int_0^1 \frac{\dif}{\dif \gamma} \hat{f}_a(W(\gamma)) d\gamma,
\end{equation*}
it follows that $\hat{f}_a(W(1)) = f_a(S_2) \leq \hat{f}_a(W(0)) = f_a(S_1)$. Thus, $f_a$ is monotone decreasing, and $f$ is submodular by Theorem \ref{theorem:mono}.

Finally, it can be seen from additivity of the Gramian \eqref{gramadd} that $f$ is monotone increasing, which just means that adding an actuator to the system cannot decrease its controllability.
\end{proof}

\subsection{Log determinant of the Gramian}
We now consider the log determinant of the controllability Gramian. We assume again that for any $S \subseteq V$ the associated Gramian is invertible. We have the following result. 
\begin{theorem}\label{theorem:determinant}
Let $V = \{ b_1,...,b_M\}$ be a set of possible input matrix columns and $W_S$ the controllability Gramian associated with $S \subseteq V$. The set function $f\colon 2^{V}\to \mathbf{R}$, defined as
\[f(S) = \log \det W_{S}  \]
is submodular and monotone increasing.
\end{theorem}
\begin{proof} The proof uses the same idea as before, namely, showing that the derived set functions $f_a: 2^{V-\{a\}} \to \mathbf{R}$
\begin{align*}
 f_a(S) &=  \log \det  W_{S \cup \{a\} } - \log \det W_{S} \\
 &=  \log\det (W_{S}+W_{a}) - \log\det W_{S} 
\end{align*}
are monotone decreasing for any $a \in V$. For arbitrary $a \in V$, $S_1\subseteq S_2 \subseteq V-\{a\}$, define again $W(\gamma) = W_{S_1} + \gamma(W_{S_2} - W_{S_1})$ for $\gamma \in [0,1]$ and
\begin{equation*}
\hat{f}_a(W(\gamma)) =  \log \det (W(\gamma) + W_a ) - \log \det W(\gamma).
\end{equation*}
We have
\begin{align*}
\frac{\dif}{\dif \gamma} &\hat{f}_a(W(\gamma)) \\
 &=\frac{\dif}{\dif \gamma}\left[ \log \det (W(\gamma) + W_a) - \log \det W(\gamma) \right] \\
 &=\tr[(W(\gamma)+W_{a})^{-1}(W_{S_2} - W_{S_1})] \\ &\quad -\tr[W(\gamma)^{-1}(W_{S_2} - W_{S_1})]\\
 &=\tr[((W(\gamma)+W_{a})^{-1}-W(\gamma)^{-1})(W_{S_2} - W_{S_1})]\\
 &\leq 0.
\end{align*}
We used the matrix derivative formula $\frac{\dif}{\dif \gamma} \log \det X(\gamma) = \tr [X(\gamma)^{-1} \frac{\dif}{\dif \gamma} X(\gamma)]$ \cite{petersen2008matrix} to obtain the second equality. The remainder of the proof follows the previous proof.
%
\end{proof}

\vspace{\baselineskip}
\begin{corollary}
The related set function $g: 2^V \rightarrow \mathbf{R}$ defined by $g(S)=\log\sqrt[n]{\det{W_S}}$ is submodular and monotone increasing.
\end{corollary}
\begin{proof}
We have
\[g(S)= \frac{1}{n} \log \det W_S.\]
Thus, from Theorem \ref{theorem:determinant} $g$ is a non-negatively scaled version of a submodular and monotone increasing function and therefore also submodular and monotone increasing.
\end{proof}

Not all directions in the state space may be of equal importance, one might want to use a weight matrix as an additional design parameter for an actuator selection problem. In a simple case, the weight matrix could be a diagonal matrix, assigning a relative weight to every state. We have the following corollary; the proof follows exactly the same arguments as in the previous theorems.

\begin{corollary}
Let $V = \{ b_1,...,b_M\}$ be a set of possible input matrix columns and $W_S$ the controllability Gramian associated with $S \subseteq V$. The set functions  $g_1,g_2 \colon 2^V \to \mathbf{R}$ defined as
\begin{equation*}
\begin{aligned}
g_1(S) &= \log\det(C W_S C^T), \\
g_2(S) &= - \tr [(C W_S C^T)^{-1}],
\end{aligned}
\end{equation*}
where  $C \in\mathbf{R}^{p \times n}$ with $p\leq n$ and $\rk(C)=p$, are submodular and monotone increasing.
\end{corollary}


\begin{remark}
Interestingly, other combinatorial network design problems unrelated to controllability have a strikingly similar mathematical structure. Specifically, in \cite{shames2014rigid} it is shown that problems in which one chooses sets of nodes or edges to optimize certain rigidity properties of a network, which relate to formation control and network localization objectives, are also submodular set function optimization problems for which greedy algorithms yield solutions with suboptimality guarantees. In that setting, one can define a ``rigidity Gramian'' to quantify desirable rigidity properties, and the results and proofs techniques are nearly identical to what we present in Theorems 4-6. Furthermore, problems in which one adds sets of edges to a network to optimize the coherence of the resulting network, which relates to robustness of consensus process to additive noise, also have a similar structure \cite{summers2014topology}. 
\end{remark}

\subsection{Rank of the Gramian}
The controllability metrics $-\tr(W_S^{-1})$ and $\log \det W_S$ fail to distinguish amongst subsets of actuators that do not yield a fully controllable system. In particular, these functions are undefined, or are interpreted to return $-\infty$, when the Gramian is not full rank. One way to handle cases where the controllability Gramian in not invertible is to consider its rank. The following result shows that this is also a submodular set function. 
\begin{theorem}
Let $V = \{ b_1,...,b_M\}$ be a set of possible input matrix columns and $W_S$ the controllability Gramian associated with $S \subseteq V$. The set function $f\colon 2^V \to \mathbf{R}$, defined as
\[f(S) = \rk(W_S) \]
is submodular and monotone increasing.
\end{theorem}

\textit{Proof:} For two linear transformations $V_1$,~$V_2 \in \mathbf{R}^{n\times n}$, we have
\begin{equation*}
\begin{aligned}
\rk & (V_1+V_2) \\ &= \rk(V_1) + \rk(V_2) -\dim(\text{range}(V_1) \cap \text{range}(V_2)).
\end{aligned}
\end{equation*}
We can form gain functions $f_a: 2^{V-\{a\} } \rightarrow \mathbf{R}$
\begin{equation}
\begin{aligned}
f_a(S)&=\rk(W_{S \cup \{a\} }) - \rk (W_S) \\
           &=\rk(W_{a}) - \dim(\text{range}(W_S)\cap\text{range}(W_{a})).
\end{aligned}
\end{equation}
It is now easy to see that $f_a$ is monotone decreasing: the first term in the second line is constant and the second term decreases because $\dim (\text{range}(W_S))$ only increases with $S$. That $f$ is monotone increasing is clear from additivity of the Gramian \eqref{gramadd}.
\hfill \qed

Note that Olshevsky has analyzed a greedy algorithm for maximizing the rank of the controllability matrix, though not in a submodularity framework \cite{olshevsky2014minimal}.

Another way to handle uncontrollable systems is to work with related continuous metrics defined for uncontrollable systems, such as the trace of the pseudoinverse $\tr (W_S^\dag)$, which corresponds to the average energy required to move the system around the controllable subspace, or the log product of non-zero eigenvalues $\log\Pi_{i=1}^{\rk W_S} \lambda_i(W_S)$, which relates to the ``volume'' of the subspace reachable with one unit of input energy.

\subsection{Smallest eigenvalue of the Gramian}
We have seen so far that the trace of the Gramian is a modular (and thus both sub- and supermodular) set function of actuator subsets and that the trace of the inverse Gramian, the log determinant of the Gramian, and the rank of the Gramian are submodular set functions. The first two functions are also concave matrix functions. Given the connections between submodular functions and concave functions, one might be tempted to conjecture that any concave function of the Gramian corresponds to a submodular function of actuator subsets. However, we now show by counterexample that this is false. Consider the set function $f(S) = \lambda_{min}(W_S)$, which corresponds to the concave matrix function that returns the smallest eigenvalue of the Gramian.

We show an example where this function violates the diminishing gains property \eqref{submod1} of a set function $f(S)$
\begin{equation*}\label{eq:supermodular}
\dgp{s}{A}\geq\dgp{s}{B}, \quad A\subseteq B \subseteq V,\; s\notin B,
\end{equation*}
where $\dgp{s}{A} = f(A \cup \{s\}) - f(A).$ Consider the system defined by
\[A = \begin{bmatrix} -8 &0 & -2\\
0 & -2 & -8\\
7& 0 & -3
\end{bmatrix},\qquad B_V=[b_V]=I_3.\]
We see that the diminishing returns  property holds in some cases, e.g.,
\[\dgp{b_3}{\{b_1\}}=0.037 \geq 0.033 = \dgp{b_3}{\{b_1,b_2\}},\]
but is violated in others
\[\dgp{b_3}{\{b_2\}}=0.001 \leq 0.033 = \dgp{b_3}{\{b_1,b_2\}}.\]

\subsection{Dynamic network centrality measures}
Network centrality measures are real-valued functions that assign a relative ``importance'' to each node within a graph. Examples include degree, betweenness, closeness, and eigenvector centrality \cite{newman2010networks}. The meaning of importance and the relevance of various metrics depends highly on the modeling context. For example, PageRank, a variant of eigenvector centrality, turns out to be a much better indicator of importance than vertex degree in the context of networks of web pages, one of the core factors leading to Google's domination of web search.

In the context of complex dynamical networks, the controllability metrics described above can be used to define a control energy-based centrality measures, describing the importance of a node in terms of its ability to move the system around the state space with a low-energy time-varying control input. In particular, given a system dynamics matrix $A \in \mathbf{R}^{n \times n}$, imagine that it is possible to place an actuator at each individual node in the network; thus, define $V = \{e_1,...,e_n \}$, where $e_i$ is the standard unit basis vector in $\mathbf{R}^n$, i.e., $e_i$ has a $1$ in the $i$th entry and zeros elsewhere. We define several Control Energy Centrality measures for a complex dynamical network as follows.
\begin{definition}[Control Energy Centralities]
Given a complex network with $n$ nodes and an associated stable linear dynamics matrix $A \in \mathbf{R}^{n\times n}$, we define the following Control Energy Centrality measures for each node $i$
\begin{itemize}
\item \textbf{Average Controllability Centrality}
\begin{equation}
C_{AC}(i) = \mathbf{tr}(W_i), \quad i \in V
\end{equation}
\item \textbf{Average Control Energy Centrality}
\begin{equation}
C_{ACE}(i) = -\mathbf{tr}(W_i^\dag), \quad i \in V
\end{equation}
\item \textbf{Volumetric Control Energy Centrality}
\begin{equation}
C_{VCE}(i) = \log \prod_{j=1}^{\rk W_i}  \lambda_j (W_i), \quad i \in V,
\end{equation}
\end{itemize}
where $W_i$ is the infinite-horizon controllability Gramian that satisfies $AW_i + W_i A^T + e_i e_i^T = 0$.
\end{definition}
These measures provide more relevant quantities of centrality than purely graph based measures in the context of dynamical systems and control, and can give quite a different view of what nodes are important. The greedy algorithm for choosing nodes in which to inject control signals can be interpreted as choosing the most central node at each iteration, given the current set of controlled nodes.
An interesting topic for future work would be to explore the distribution of the Control Energy Centrality measure in random networks and networks from various application domains.

Pasqualetti \emph{et al.} have also defined a different network centrality measure based on the controllability Gramian \cite{pasqualetti2014controllability}. In the context of networks with consensus dynamics, Chapman and Mesbahi have also defined a related network centrality measure that quantifies effectiveness of each agent in tracking the mean of a noisy signal \cite{chapman2013semi}. It is possible to define many other network centrality measures related to network dynamics and control, e.g., based on the leader selection metrics of Clark \emph{et al.} \cite{bushnell2014supermodular,clark2012leader}.

\subsection{Computational scaling for large networks}
In this subsection we discuss computational techniques that can be used to scale the greedy algorithm described in Section III-A to large structured networks. First, specialized algorithms can be used to exploit sparsity and compute low rank solutions to Lyapunov equations. In particular, when computing the Gramian associated with an individual actuator, the Cholesky factor-alternating direction implicit algorithm of Li and White \cite{li2002low} allows one to exploit both the rank-one structure of the constant term in the Lyapunov equation (viz., $bb^T$) and any sparsity structure in the network dynamics matrix $A$. Moreover, it is often the case in large networks that the Gramian associated with an individual actuator is low rank or approximately low rank itself. One can obtain further computational benefits by computing low rank approximations of these Gramians also using methods from \cite{li2002low}.

Second, several techniques can be used to improve the greedy algorithm. Each iteration in the standard version is trivially parallelizable. The Gramians associated with each possible actuator can be pre-computed using the specialized methods mentioned above independently and in parallel. Because the Gramian is additive in the actuators, effectively one can solve the Lyapunov equation for any set of actuators by solving it in parallel for each individual actuator and summing the results. Then at each iteration, the marginal gain of each actuator can also be computed in parallel by adding its Gramian to the Gramian of the current set of added actuators and evaluating the metric (trace, logdet, etc.). When the individual Gramians are low rank, the marginal gains can be computed more efficiently by using low rank update formulae, e.g., the Sherman-Morrison formula for trace of the inverse Gramian or the matrix determinant lemma for log determinant of the Gramian. 

Alternatively, there is also an accelerated version of the greedy algorithm, originally due to Minoux \cite{minoux1978accelerated}, in which one can exploit the submodularity of the set functions to significantly reduce the number of times that marginal gains for the actuators need to be computed. This can lead to orders of magnitude speedups in practice; see, e.g., \cite{krause2012submodular}.

\section{Numerical Examples}\label{sec:examples}
In this section, we illustrate the results on randomly generated systems and on a problem of power electronic actuator placement in a model of the European power grid. The problem data is a system dynamics matrix $A \in \mathbf{R}^{n\times n}$, a set of possible input matrix columns $V = \{ b_1,...,b_M \}$, and an integer number $k$ of actuators to choose from this set to form an input matrix that maximizes a controllability metric.

%

\subsection{Greedy performance on random systems}
To evaluate performance of the greedy algorithm and to compare the various controllability metrics, we first consider randomly generated data. We use Matlab's \texttt{rss} routine to generate a stable dynamics matrix with random stable eigenvalues. We use $V=\{e_1,...,e_n\}$, where $e_i$ is the $i$th unit vector in $\mathbf{R}^n$, i.e., we assume one can choose states in which a control input can be injected.


Figure~\ref{fig:greedy_spread} shows the result of applying the greedy algorithm to maximize the log determinant metric with $n=25$ and $k=7$. This problem is small enough to evaluate every possible 7-element actuator subset, and this result is also shown in a histogram, shifted so that $\min_S f(S) = 0$. The support of $f(S)$ for $|S| = k$, is relatively narrow and close to the optimal value. Hence, the greedy bound is not informative in this case, as $63\%$ of the optimum is lower than the values achieved by all of the size $k$ subsets. Nevertheless, our algorithm finds a good set $S_{greedy}$ scoring
\[\frac{f(S_{greedy})}{f(S_{opt})}\approx99\%\]
of the optimum value $f(S_{opt})$, where $S_{opt}$ is an optimal subset, and is better than 99.93\% of all other $n$ choose $k$ possible choices. We repeated the greedy algorithm for 500 randomly generated stable dynamic matrices and found that in all cases the greedy algorithm returned a selection better than 99.5\% of all possible choices. In other words, for this example the greedy algorithm provides a near-optimal selection and also one that performs much better than the worst case bound.


We also compare for this example the four continuous metrics $\tr(W_S), \tr(W_S^{-1}), \log \det W_S$, and $\lambda_{min}(W_S)$. Figure \ref{fig:greedy_comp} shows the eigenvalues of the Gramians resulting from applying the greedy algorithm with each metric. The results are averages over 10,000 random samples of stable dynamics matrices. For the trace metric, Theorem \ref{contrmod} guarantees that the greedy algorithm finds the globally optimal subset. We can see that this metric tends to focus on making the largest few eigenvalues large at the expense of the smaller eigenvalues. In contrast, the trace of the inverse Gramian and the log determinant strike a compromise, with both resulting in similar eigenvalue distributions. The largest eigenvalues are not as large as the ones achieved by optimizing the trace metric, but they do better on average on the smaller eigenvalues. Although the globally optimal subset is not guaranteed to be found, the submodularity of these metrics guarantees that a near optimal subset is produced by the greedy algorithm, as proved in Theorems \ref{theorem:trace} and \ref{theorem:determinant}. On the other hand, the final metric focuses exclusively on the smallest eigenvalue, but for this example actually does slightly worse on average than the trace of the inverse Gramian on the smallest eigenvalue.  This may result from the fact that the smallest eigenvalue metric is not submodular, and so thus is not always guaranteed to produce a near optimal selection. However, even using the greedy algorithm on this non-submodular metric does not do too much worse than trace of the inverse Gramian and the log determinant on other eigenvalues, and does better than trace of the Gramian on most of the smaller eigenvalues.

\begin{figure}
\centering
\includegraphics[width=\linewidth,height=4.5cm]{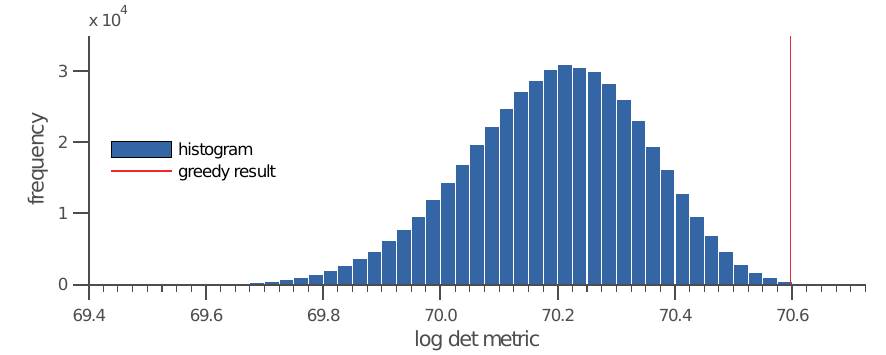}
\caption{\label{fig:greedy_spread}Histogram displaying the shifted log determinant metric for all possible selections of 7 actuators from a set of 25. The result achieved by greedy optimization is displayed by the red line, which is better than 99.93\% of all other selections.}
\end{figure}

\begin{figure}
\centering
\includegraphics[width=\linewidth,height=5cm]{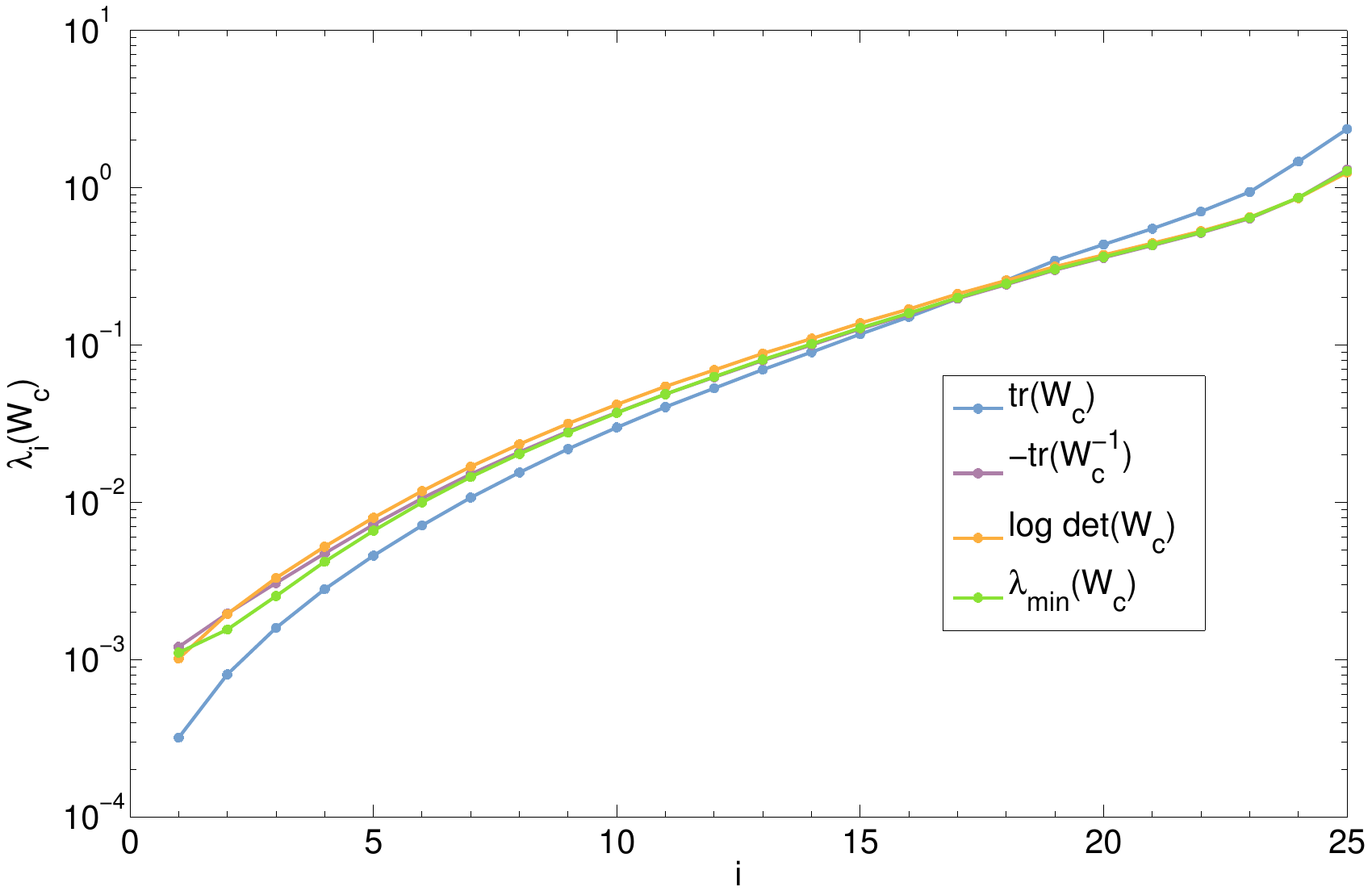}
\caption{\label{fig:greedy_comp} Eigenvalues of the Gramian averaged over 10,000 random samples of stable dynamics matrices for several continuous metrics resulting from applying the greedy algorithm to select 7 actuators from a set of possible 25.}
\end{figure}

%
%

\subsection{Power electronic actuator placement in the European power grid}
Recently developed power electronic devices, such as high voltage direct current (HVDC) links or flexible alternating current transmission devices (FACTS), can be used to improve transient stability properties in power grids by modulating active and reactive power injections to damp frequency oscillations and prevent rotor angle instability \cite{fuchs2011}. In this section, we illustrate our results by using them to place such power electronic actuators in a model of the European power grid. We emphasize that this section is intended to show that there are practical and relevant problems that could be studied using the theory in the preceding sections; however, many important political and economic issues are neglected, and placements are evaluated purely from a controllability perspective. 

We consider a simplified model of the European grid derived from \cite{Haase2006} with 74 buses, each of which is connected to a generator and a constant impedance load. We consider the placement of HVDC links, which are modeled as ideal current sources that can instantaneously inject AC currents into each of their terminal buses; for modeling details see \cite{fuchs2011,fuchs2013a,fuchs2013b}. The system dynamics we consider here are based on the swing equations, a widely-used nonlinear model for the time evolution of rotor angles and frequencies of each generator in the network \cite{kundur1994power}. Each HVDC link has three degrees of freedom that allow influence of the frequency dynamics at the corresponding buses. The nonlinear model is linearized\footnote{Ideally, one would of course want to evaluate actuator placement on the nonlinear model, but even evaluating controllability metrics can be extremely difficult computationally, even for small-scale nonlinear systems. This section is intended to illustrate the theory from the previous section, so we focus on a linearized model, though actuator placement problems for nonlinear networks are an important topic for future work.} about a desired operating condition for each possible HVDC link placement, and the placements are evaluated based on the linearized model using the infinite-horizon controllability Gramian. In principle one can easily work with the finite-horizon Gramian; again we chose to use the infinite-horizon due to the simplicity of its computation, and the results are qualitatively similar when the finite-horizon Gramian is used.

Each generator has two associated states: rotor angle and frequency, which gives a 148-dimensional state space model, i.e., $A \in \mathbf{R}^{148\times 148}$, which always turns out to be stable. Since an HVDC link could be placed in principle between any two distinct nodes in the network, there are 2701 possible locations. Consider the problem of finding the best subset of size 10. This gives approximately $5.6\times 10^{27}$ possible combinations, far too many for a brute force search. 

Figure \ref{grid1} shows the network and the 10 best placements according to the controllability Gramian trace metric with all state space directions weighted equally, i.e., $C = I_{148}$. The best three are relatively long lines connecting the northeast-southwest and northwest-southeast quadrants of the network, respectively. A modal analysis of the dynamics matrix reveals that these choices correlate well with directions associated with lightly damped modes in the rotor angle dynamics. The next group of placements is concentrated in the southeast, indicating that there is room to improve control authority by increasing connectivity in this sparsely connected region. This also indicates a potential weakness in the trace metric, which may cluster actuators to get high controllability in a few of the more controllable directions at the expense of controllability in other directions. Figure \ref{fig:dist} shows the sorted values of the metric, with the top few placements giving a substantial benefit over other placements. 

%

\begin{figure} 
\begin{center} 
\includegraphics[width=\linewidth]{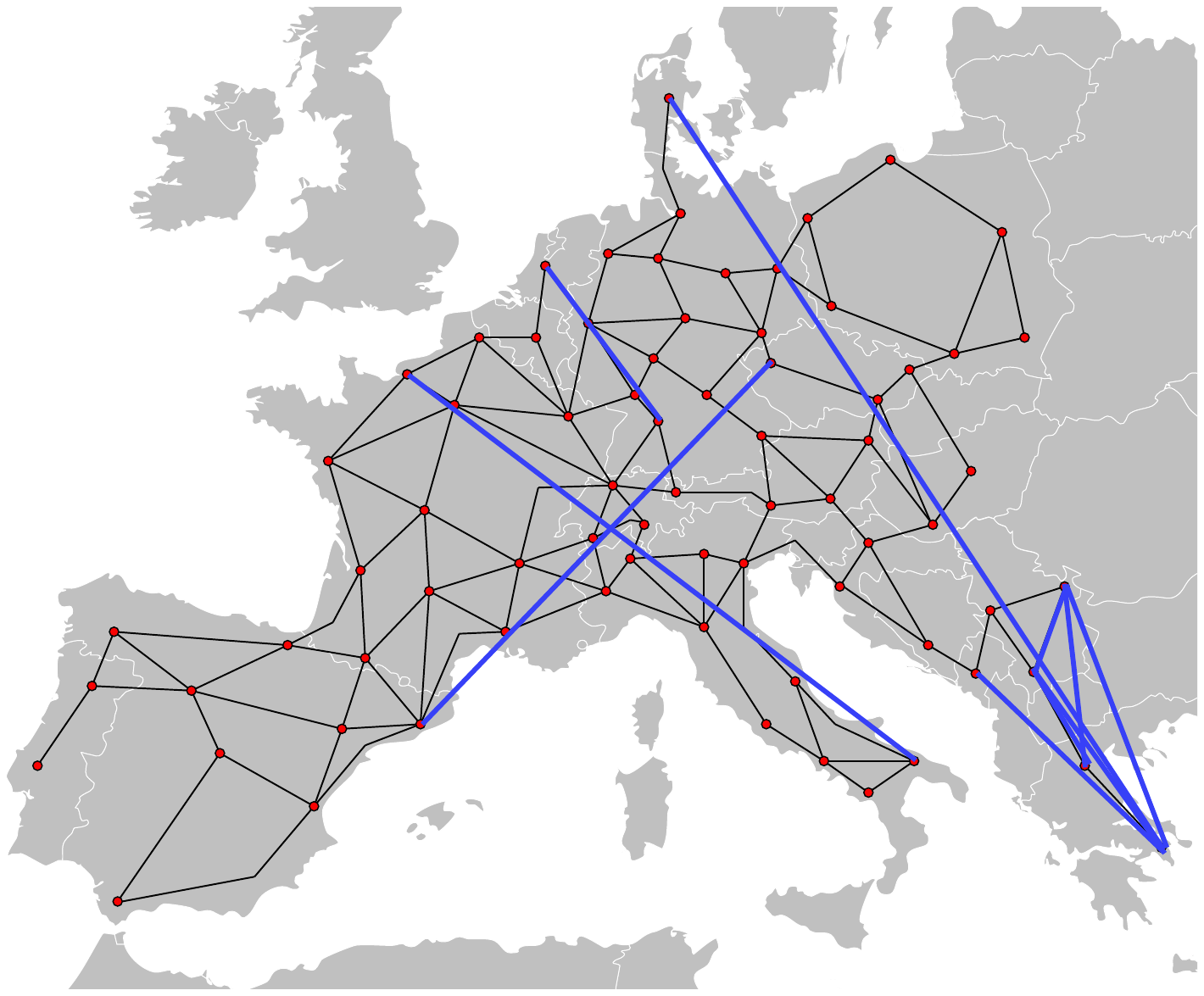}
\caption{Network of the 74-bus European grid model. The red dots show the buses, and the black lines between buses show normal AC transmission lines. The best 10 HVDC line placements according to the controllability Gramian trace metric are shown by the bold blue lines.} \label{grid1}
\end{center}
\end{figure}

\begin{figure} 
\begin{center} 
\resizebox{0.89\linewidth}{!}{\includegraphics{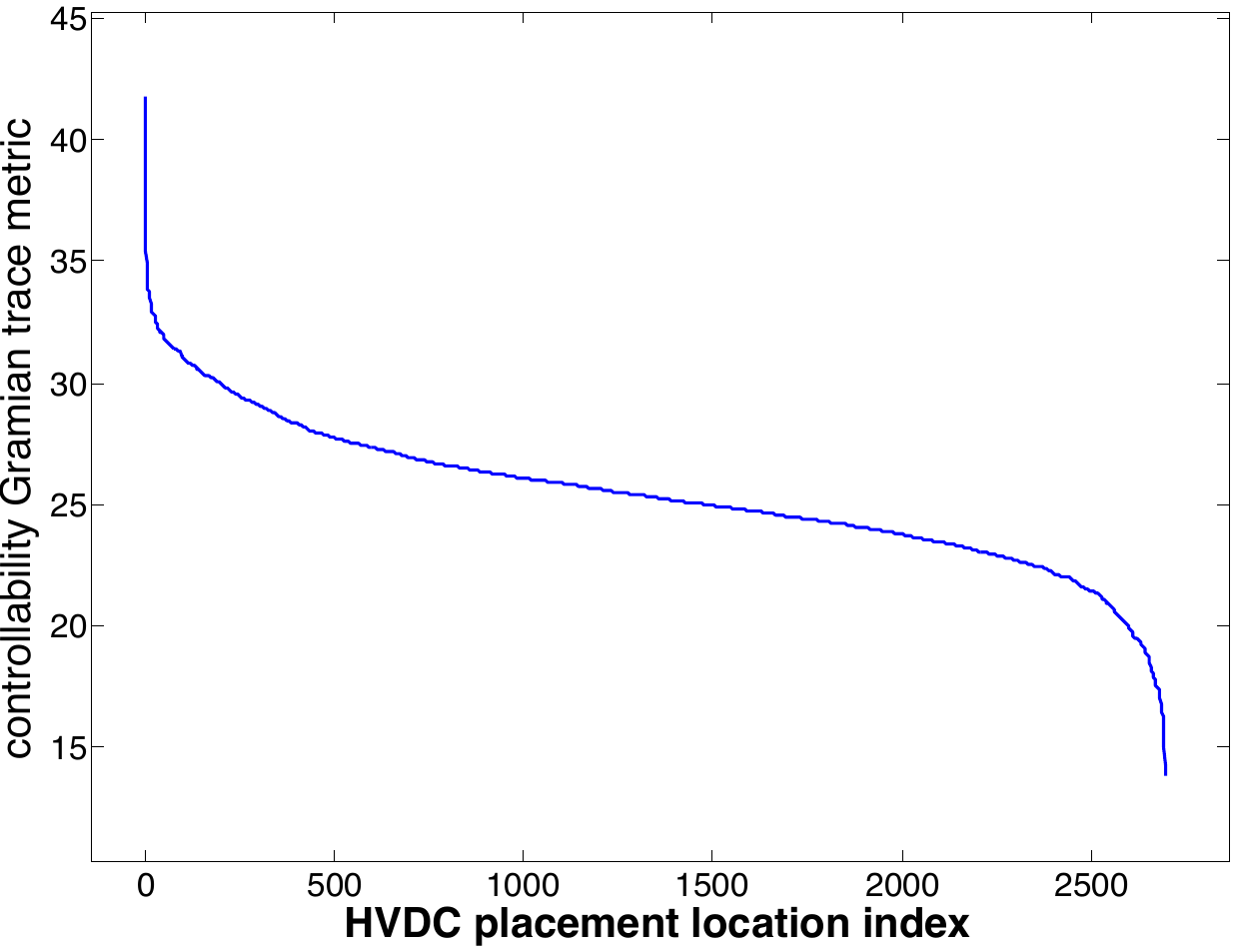}}
\caption{Sorted values of the controllability Gramian trace metric. The vertical axis shows the amount each particular actuator placement adds to the trace of the Gramian. The optimal value is the sum of the first 10 amounts; the top few placement give substantial benefit over other placements.} \label{fig:dist}
\end{center}
\end{figure}

\begin{figure}
        \centering
                \includegraphics[width=\linewidth]{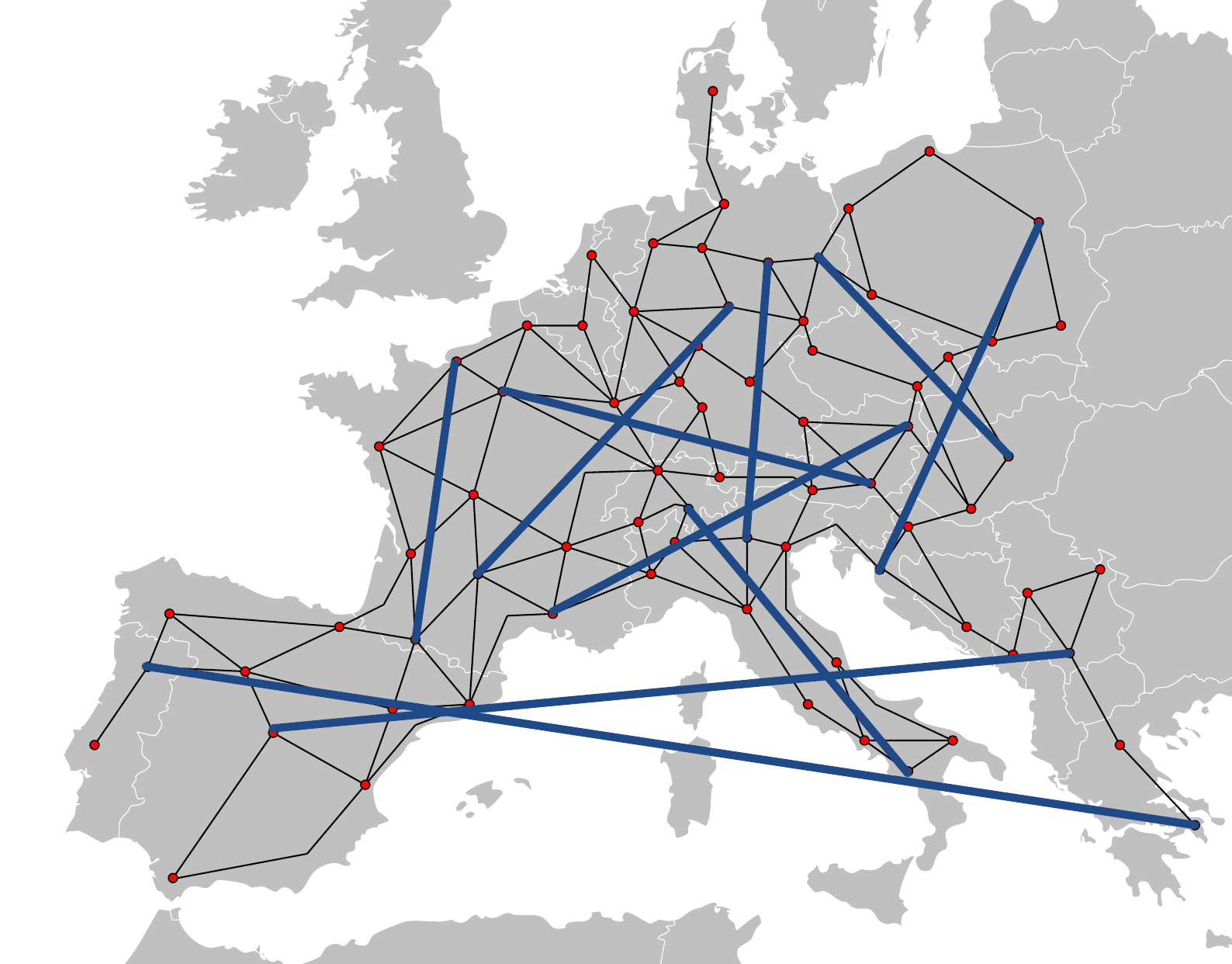}
                \caption{Network of the 74-bus European grid model.. The best 10 HVDC line placements achieved by the greedy algorithm maximizing the log determinant of the controllability Gramian are shown by the bold blue lines.}
                \label{fig:map_res}
\end{figure}

Figure \ref{fig:map_res} shows the placement obtained by using the greedy algorithm with the log determinant metric, using the rank metric until the system becomes controllable. Compared to the trace metric, we see that the lines are in general longer, connecting buses that are further apart, and more evenly distributed in the network, and no node is part of more than one HVDC line. These placements can also be seen to align with directions corresponding to lightly damped modes in the rotor angle dynamics, though with a different distribution across the modes than with the trace. Although both metrics tend to produce placements with a similar qualitative function, the two sets of obtained placements are quite different. 
\section{Conclusions and Outlook}
We have considered optimal actuator placement problems in complex dynamical networks. These problems are in general difficult combinatorial optimization problems; however, we have shown that an important class of metrics related to the controllability Gramians yield modular and submodular set functions. This allows globally optimal or near optimal placements to be obtained with a simple greedy algorithm. By duality, all of the results hold for corresponding sensor placement problems using metrics of the observability Gramian. To our knowledge, this is the first such investigation of submodularity in the context of controllability and observability in dynamical systems. We also defined several dynamic Control Energy Centrality measures, which assigns an importance value to each node in a dynamical network based on its ability to move the system around the state space with a low-energy time-varying control input. The results were illustrated via placement of power electronic actuators in a model of the European power grid. 

There are many open problems involving the structure of combinatorial optimization problems in the optimal placement of sensors and actuators in complex networks. For example, there are many other quantitative metrics of controllability and observability, such as those associated with optimal control and filtering design problems, that may be more appropriate in certain settings. Nothing is known about modularity or submodularity for any other metrics. Further future work involves exploring other case studies in power networks, biological networks, social networks, and discretized models of infinite-dimensional systems. For more complicated system models, such as constrained, nonlinear, hybrid, etc., corresponding controllability questions are much more complicated, and the available tools do not scale well computationally, but one could explore how efficient methods could be used to obtain approximate metrics in these types of systems.  Finally, an important and interesting topic for future work is to investigate how various graphical properties of the network structure affect the actuator placement results, which may lead to insights about energy-related controllability for complex dynamical networks. 


\section*{Acknowledgements}
The authors would like to thank Dr. Alex Fuchs for providing details and helpful discussion about the power system model discussed in Section IV.

\bibliographystyle{plain}  
\bibliography{refs2}

\end{document}